\documentclass[11pt,reqno]{amsart}

\usepackage[usenames,dvipsnames]{color}
\newcommand{\red}[1]{{\color{red} #1}}
\newcommand{\blue}[1]{{\color{blue} #1}}

\usepackage{amsmath}
\usepackage{amssymb, latexsym, amsfonts, amscd, amsthm, mathrsfs, enumerate, esint}

\newcommand{\rc}{\kappa(x,y)}

\newtheorem{thm}{Theorem}[section]
\newtheorem{cor}[thm]{Corollary}
\newtheorem{lem}{Lemma}[section]

\theoremstyle{definition}
\newtheorem{defn}{Definition}[section]
\newtheorem{remark}{Remark}[section]
\newcommand{\R}{\mathbb{R}}

\newcommand{\G}{\mathbb{G}}
\newcommand{\cA}{\mathcal{A}}

\newcommand{\sD}{\mathscr{D}}

\newcommand{\sF}{\mathscr{F}}
\newcommand{\sH}{\mathscr{H}}
\newcommand{\sL}{\mathscr{L}}
\newcommand{\sP}{\mathscr{P}}

\DeclareMathOperator{\diam}{diam}

\usepackage{graphicx}

\usepackage{amssymb,amsfonts,amsmath}

\begin{document}




\title{Exact and Asymptotic Results on Coarse Ricci Curvature of Graphs}

\author[Bhattacharya]{Bhaswar B. Bhattacharya}
\address{Department of Statistics, University of Pennsylvania, Philadelphia, USA,
{\tt bhaswar@wharton.upenn.edu}}

\author[Mukherjee]{Sumit Mukherjee}
\address{Department of Statistics, Columbia University, New York, USA, {\tt  sm3949@columbia.edu}}

\begin{abstract}  
Ricci curvature was proposed by Ollivier in a general framework of metric measure spaces, and it has been studied extensively in the context of graphs in recent years. In this paper we prove upper bounds for Ollivier's Ricci curvature for bipartite graphs and for the graphs with girth at least 5. We also prove a general lower bound on the Ricci curvature in terms of the size of the maximum matching in an appropriate subgraph.  As a consequence, we characterize the Ricci-flat graphs of girth 5. Moreover, using our general lower bound and the Birkhoff-von Neumann theorem, we give a necessary and sufficient condition for the structure of Ricci-flat regular graphs of girth 4. Finally, we obtain the asymptotic Ricci curvature of random bipartite graphs $G(n,n, p)$ and random graphs $G(n, p)$, in various regimes of $p$.
\end{abstract}

\subjclass[2010]{90C08, 05C70, 52B12}
\keywords{Graph curvature, Optimal transportation, Random graphs, Wassertein's distance.}

\maketitle


\section{Introduction}

Ricci curvature is a fundamental concept in Riemannian geometry, which provides a way of measuring the degree to which the geometry determined by a given Riemannian metric might differ from that of $\mathbb R^n$.  Ricci curvature plays an important role in general relativity, where it is the key term in the Einstein field equations, and in the celebrated Ricci flow equation, where a time-dependent Riemannian metric is deformed in the direction of its negative its Ricci curvature.  Bakry and \'Emery \cite{bakryemery} attempted to define Ricci curvature through the heat semigroup on a metric measure space. In the recent years, there have been several work on defining a synthetic Ricci curvature on general metric measure spaces by Sturm \cite{sturmI,sturmII}, Lott and Villani \cite{lottvillani}, and Ohta \cite{ohta}.

In the context of graphs, Chung and Yau \cite{chungyau} developed the notion of Ricci-flat graphs, while proving log-Sobolev inequalities. Later, Lin and Yau \cite{linmrl} generalized the notion of Bakry and \'Emery to the framework of graphs. Finally, Ollivier \cite{ollivierthesis} introduced a notion of coarse Ricci curvature that extends to Markov chains on metric spaces. This was used to generalize a series of classical theorems in positive Ricci curvature, such as spectral gap estimates, concentration of measure or log-Sobolev inequalities \cite{ollivier,olliviersurvey}. Joulin and Ollivier \cite{joulinollivier} proved nonasymptotic estimates for the rate of convergence of empirical means of Markov chains, together with a Gaussian or exponential control on the deviations of empirical means, under the assumption of positive curvature of the underlying space. This assumption reduces to the well-known contraction under path coupling when the underlying space is a finite graph, which has been used extensively to prove fast mixing of several discrete Markov chains (refer to Chapter 14 of Levin et al. \cite{peresbook} for details on path coupling and its application to fast mixing and approximate counting of proper $q$-colorings of a graph).  In Riemannian geometry, both Ollivier's Ricci curvature and Bakry-\'Emery curvature-dimension inequality gives lower bound estimates for the first eigenvalue of the Laplace operator \cite{bakry,bauer_jost_liu,ollivier}. For a detailed exposition on geometric methods concerning eigenvalue estimates in the study of Markov chains refer to \cite{diaconis_mcmc,diaconis_stroock,fulman} and the references therein.

Recently, Ollivier's Ricci curvature has been studied in the context of graphs by Jost and Liu \cite{jost}, Paeng \cite{paeng}, and Cho and Paeng \cite{chopaeng}. Very recently, Loisel and Romon \cite{loisel_romon} described a method for computing  Ollivier's Ricci curvature for polyhedral surfaces and discussed the connections with linear programming. Lin et al. \cite{ricciflat,tohoku} considered a modified definition of Ricci curvature of graphs, and proved several analogous results. In this paper we study Ollivier's Ricci curvature on graphs using the Markov kernel of the simple random walk on the graph. We begin with upper bounds for the Ricci curvature for bipartite graphs and for the graphs with girth at least 5. 
Using these results we characterize Ricci-flat graphs of girth at least 5. We also prove several other bounds on the Ricci curvature, involving different graph parameters. The most interesting among them is the derivation of a general lower bound in terms of the size of the maximum matching in an appropriate neighborhood subgraph. Using this we give a necessary and sufficient condition on the structure of Ricci-flat regular graphs of girth 4. Finally, using some results from approximate matching in random graphs, we study the asymptotic behavior of Ricci curvature of random bipartite graphs $G(n,n, p)$ and random graphs $G(n, p)$, in various regimes of $p$. 


\subsection{Ollivier's Ricci curvature: Definitions and Notations}

In this section we recall some basic facts about Ollivier's Ricci curvature on graphs and introduce other relevant definitions and notation.

For two probability measures $\mu_1, \mu_2$ on a metric space $(X, d)$, the {\em transportation distance} (or the {\em Wasserstein distance}) between them is defined as
\begin{equation}
W_1(\mu_1, \mu_2)=\inf_{\nu \in M(\mu_1, \mu_2)}\int_{X\times X}d(x, y)\mathrm d\nu(x, y),
\label{eq:wd}
\end{equation}
where $M(\mu_1, \mu_2)$ is the collection of probability measures on $X\times X$ with marginals $\mu_1$ and $\mu_2$.
Another useful representation of the transportation distance is through the celebrated Kantorovich duality (Theorem 1.14, Villiani \cite{villani}), which states that 
\begin{equation}
W_1(\mu_1, \mu_2)=\sup_{f, 1-Lip}\left\{\int_X f\mathrm d\mu_1-\int_X f\mathrm d\mu_2\right\},
\label{eq:wdduality}
\end{equation}
where the supremum is taken over all functions $f:X\rightarrow \mathbb R$ which satisfy $|f(x)-f(y)|\leq d(x, y)$, for all $x, y \in X$.


The transportation distance between probability measure is used to define the Ricci curvature of metric measure spaces. A metric measure space $(X, d, m)$ is a metric space $(X, d)$, and a collection of probability measures $m=\{m_x: x \in X\}$ indexed by the points of $X$. The coarse Ricci curvature of a metric measure space is defined as follows:

\begin{defn}[Ollivier \cite{ollivierthesis}]On any metric measure space $(X, d,m)$, for any two distinct points $x, y \in X$, the {\em coarse Ricci curvature} of $(X, d, m)$ of $(x, y)$ is defined as $\kappa(x, y) := 1-\frac{W_1(m_x, m_y)}{d(x, y)}$.
\label{def:ollivierdef}
\end{defn}

Hereafter, we shall refer to Ollivier's coarse Ricci curvature simply as Ricci curvature and we shall study it for locally finite graphs. Consider a locally finite weighted simple graph $G=(V(G), E(G))$, where each edge $(x, y)\in E(G)$ is assigned a positive weight $w_{xy}=w_{yx}$. The graph is equipped with the standard shortest path graph distance $d_G$, that is, for $x, y \in V(G)$, $d_G(x,y)$ is the length of the shortest path in $G$ connecting $x$ and $y$.  The {\it girth} of $G$ is the length of the shortest cycle in $G$. For $x\in V(G)$ define the degree $d_x=\sum_{(x, y)\in E(G)}w_{xy}$ and the neighborhood $N_G(x)=\{y\in V(G): (x, y)\in E(G)\}$. For each $x \in V(G)$ define a probability measure 
$$m_x(y)=\left\{
\begin{array}{cc}
\frac{w_{xy}}{d_x}, &  \hbox{if } y\in N_G(x)   \\
0, &  \hbox{otherwise.}
\end{array}
\right.
$$
Note that these are just the transition probabilities of a weighted random walk on the vertices of $G$. If $m_G=\{m_x: x \in V(G)\}$, then considering the metric measure space $\mathcal M(G):=(V(G), d_G, m_G)$, we can define the Ricci curvature for any edge $(x, y)\in E(G)$ as $\kappa_G(x,y):=1-W_1^G(m_x, m_y)$. Applying Equation (\ref{eq:wd}) for $\mathcal M(G)$ we get 
\begin{equation}
\label{rcdef2}
W_1^G(m_x, m_y)=\inf_{\nu\in \mathcal{A}}\sum_{z_1\in N_G(x)}\sum_{z_2\in N_G(y)}\nu(z_1,z_2)d(z_1,z_2),
\end{equation}
where $\mathcal{A}$ denotes the set of all $d_x\times d_y$ matrices with entries indexed by $N_G(x)\times N_G(y)$ such that $\nu(x',y')\ge 0$, $\sum_{z\in N_G(y)}\nu(x',z)=\frac{w_{xx'}}{d_x}$, and $\sum_{z\in N_G(x)}\nu(z,y')=\frac{w_{yy'}}{d_y}$, for all $x'\in N_G(x)$ and $y'\in N_G(y)$. Intuitively, the Wasserstein distance measures the optimal cost to move one pile of sand to another one with the same mass. For a matrix $\nu\in \cA$, $\nu(x', y')$ represents the mass moving from $x'\in N_G(x)$ to $y'\in N_G(y)$. For this reason, the matrix $\nu$ is often called the {\it transfer plan}.

By the Kantorovich duality in Equation (\ref{eq:wdduality}) we can also write 
\begin{equation}
\label{rcdef}
W_1^G(m_x,m_y)=\sup_{f, 1-Lip}\left\{\sum_{z\in N_G(x)} f(z)m_x(z)-\sum_{z\in N_G(y)} f(z)m_y(z)\right\}.
\end{equation}

Henceforth, we denote by $\sL_1$ the set of all 1-Lipschitz functions on $G$, that is, the set of all functions
$f:V(G)\rightarrow \mathbb R$ such that $|f(x)-f(y)|\leq d_G(x, y)$, for $x, y \in V(G)$. For any $x \in V(G)$ and any function $f \in \sL_1$, define $E_x(f)=\sum_{z\in N_G(x)} f(z)m_x(z)$.  With these notation, Equation (\ref{rcdef}) now becomes $W_1^G(m_x, m_y)=\sup_{f \in \sL_1}\{E_x(f)-E_y(f)$\}. By triangle inequality \cite{ollivierthesis}, if $\kappa_G(x, y)\geq \gamma$ for all neighbors $(x, y)\in E(G)$, then $\kappa_G(x, y)\geq \gamma$, for all $x, y\in V(G)$. Therefore, it is reasonable to consider $\kappa_G(x, y)$ only for neighboring vertices $(x, y)\in E(G)$.

Hereafter, the subscript and superscript $G$ from $\kappa_G$, $W_1^G$, and $d_G$ will be often omitted when the graph is clear from the context. For notational brevity, the main theorems will be stated for unweighted graphs, that is, $w_{xy}=1$, for $(x, y)\in E(G)$. 
Finally, for $a, b\in 
\R$, define $a_{+}:=\max\{a, 0\}$, $a\wedge b:=\min\{a, b\}$, and $a\vee b:=\max\{a, b\}$.


\subsection{Prior Work on Ricci Curvature of Graphs} Recently, there has been a series of papers on coarse Ricci curvature when the metric space is a graph $G$. Jost and Liu proved the following general bounds:

\begin{thm}[Jost and Liu \cite{jost}]
\label{th:jost}
For any locally finite unweighted graph $G$, with $(x, y)\in E(G)$,
\begin{equation}
\frac{|\Delta_G(x, y)|}{d_x\vee d_y}-\left(1-\frac{1}{d_x}-\frac{1}{d_y}-\frac{|\Delta_G(x, y)|}{d_x\wedge d_y}\right)_+-\left(1-\frac{1}{d_x}-\frac{1}{d_y}-\frac{|\Delta_G(x, y)|}{d_x\vee d_y}\right)_+\le \rc \le \frac{|\Delta_G(x, y)|}{d_x\vee d_y},
\label{eq:jlcp}
\end{equation}
where $\Delta_G(x, y)$ is the number of triangles supported on $(x, y)$.
\end{thm}

They also show that the above lower bound is tight for trees. In fact, it is clear from their proof that the  lower bound is an equality whenever there are no 3, 4, or 5 cycles supported on $(x, y)$.  In particular, the lower bound is tight whenever $g(G)\geq 6$, where $g(G)$ is the girth of the graph $G$.  Recently, Cho and Paeng \cite{chopaeng} proved this independently, and also showed that the girth condition and the tree formula, obtained by putting $|\Delta_G(x,y)|=0$ in the lower bound in (\ref{eq:jlcp}), are equivalent in the following sense: the tree formula holds for all $(x, y)\in E(G)$, if and only if $g(G)\geq 6$. Cho and Paeng \cite{chopaeng} also proved other Ricci curvature bounds involving girth. In particular, they showed that if $g(G) \geq  5$, then $\rc \leq -1+ \frac{2}{\delta}$, where $\delta:=\delta(G)$ is the minimum degree in $G$, and $(x, y)\in E(G)$. They also obtained interesting lower bounds on the clique number and chromatic number of a graph with the Ricci curvature. Paeng \cite{paeng} used Ollivier's Ricci curvature to obtain upper bounds of diameter and volume for finite graphs. The relation between Ollivier's Ricci curvature and the first eigenvalue was studied by Bauer et al. \cite{bauer_jost_liu}.



Lin et al. \cite{tohoku} introduced a different notion of Ricci curvature on graphs by modifying Ollivier's definition. It is defined as the differential limit of a lazy random random walk on the graph, and we shall refer to it as the modified Ricci curvature to distinguish it from Ollivier's coarse Ricci curvature. The modified definition has some properties which are similar to the original definition, however, in several contexts they are very different. Using the modified definition, they proved a theorem on the modified Ricci curvature of the Cartesian product of graphs. They established upper bounds for diameters and the number of vertices for graphs with positive curvatures, and also proved some asymptotic properties of modified Ricci curvature for random graphs. Recently, Lin et al. \cite{ricciflat} characterized the set of all modified Ricci-flat graphs with girth at least 5, where a graph is called modified Ricci-flat whenever it has modified Ricci curvature zero for every edge in the graph. They showed that if $G$ is a connected modified Ricci-flat graph with girth $g(G) \geq  5$, then $G$ is the infinite path, or a cycle $C_n$ with $n \geq  6$, the dodecahedral graph, the Petersen graph, or the half-dodecahedral graph.

\subsection{Summary of Our Results and Organization of the Paper}

In this paper we obtain various bounds on the Ollivier's Ricci curvature. We begin with upper bounds for bipartite graphs and for the graphs with girth at least 5.  For a bipartite graph $G$ with $(x, y)\in E(G)$, we obtain a upper bound of the form $$\rc \leq -2\left(1-\frac{1}{d_x}-\frac{1}{d_y}- C_G(x, y) \right)_{+},$$
where the explicit form of $C_G(x, y)$ is given in Theorem \ref{th:bipartite}. A similar bound is obtained for graphs with girth greater than 4 (Theorem \ref{th:gfive}). As a consequence of this bound, we characterize the set of all Ricci-flat graphs of girth at least 5, where a graph $G$ is said to Ricci-flat if $\rc=0$, for all $(x, y)\in E(G)$ (Corollary \ref{cor:ricciflatfive}). This is in analogue to the result on modified Ricci curvature of Lin et al. \cite{ricciflat} in the context of Ollivier's coarse Ricci curvature.

In Theorem \ref{th:matching} we prove a general lower bound on the Ricci curvature $\rc$ in terms of the size of the matching among the non-common neighbors of $x$ and $y$ in the graph $G$. This bound is often tight, especially in regular graphs which have a perfect matching between the non-common neighbors of $x$ and $y$.  As the set of all transportation matrices in $d$-regular graphs is related to the famous Birkhoff polytope, our lower bound result combined with the celebrated Birkhoff-von Neumann theorem gives a necessary and sufficient condition on the structure of Ricci-flat regular graphs of girth 4 (Corollary \ref{cor:ricciflat4}).

Finally, we also study the Ricci curvature of random bipartite graphs $G(n,n, p)$ (Theorem \ref{bp}) and random graphs $G(n, p)$ (Theorem \ref{gnp}), in various regimes of $p$. Using a stronger version of the Hall's marriage theorem, and the existence of near-perfect matching in random bipartite graphs, we obtain the limiting behavior of the Ricci curvature in the regimes of $p$ where it has a constant limit in probability. We also show that when $np_n\rightarrow \lambda$, that is, the graph is locally tree-like, the Ricci curvature converges in distribution to the tree formula of Jost and Liu \cite{jost}. These are the first known results for Ollivier's Ricci curvature for Erd\H os-Renyi random graphs. The analogous version of these results using the modified Ricci curvature were obtained by Lin et al. \cite{tohoku}. They showed almost sure convergence to constant limits, but could not capture all the different regimes of $p$.

In Section \ref{sec:preliminaries} we prove several important lemmas which build the foundations for proving the main results. We show that the computation of Ricci curvature on a graph can be formulated as a totally unimodular linear programming problem, and so it suffices to optimize over integer valued 1-Lipschitz functions. We also prove a crucially important reduction lemma where we identify the exact neighborhood of an edge $(x, y)\in E(G)$ that needs to be considered while computing the  Ricci curvature $\kappa(x, y)$.

\section{Preliminaries}
\label{sec:preliminaries}


We begin by proving an extension lemma for Lipschitz functions on graphs. Let $G=(V(G), E(G))$ be a locally finite weighted graph. Let $U\subset V(G)$ be a fixed subset of vertices, and $d_G$ the shortest path metric on $G$. 

\begin{lem}
\label{lip}
Any 1-Lipschitz function $g:U\rightarrow\mathbb R$, that is, $|g(a)-g(b)|\le d_{G}(a,b)$, 
for $a, b\in U$, can be extended to a 1-Lipschitz function $\overline{g}:V(G)\rightarrow \mathbb R$ on $G$, that is, $|g(a)-g(b)|\le d_{G}(a,b)$,  for $a, b\in V$.
\end{lem}

\begin{proof}
Define $\overline{g}:=g$ on $U$. Let $z\in V(G)\backslash U$ be any point. Note that result follows by induction if we can construct a function $\overline g: U\bigcup \{z\} \rightarrow\mathbb R$ which is satisfies $|g(a)-g(b)|\le d_{G}(a,b)$, 
for $a, b\in U\bigcup \{z\}$. To this end, let
$$A:=\bigcap_{a\in U}[\overline{g}(a)-d_{G}(a,z),\overline{g}(a)+d_{G}(a,z)].$$
Observe that if $A$ is empty, then there must exist $a,b\in U$ such that $\overline{g}(a)+d_{G}(a,z)<\overline{g}(b)-d_{G}(b,z)$. This implies that $$g(b)-g(a)>d_{G}(a,z)+d_{G}(b,z)\ge d_{G}(a,b).$$ This contradicts the assumption that $g$ is Lipschitz on $U$, and proves that $A$ is non-empty. Therefore, we can define $\overline g(z)=r$, for some $r\in A$. Moreover, by construction $|\overline g(a)-\overline g(b)|\leq d_{G}(a, b)\leq d_{G}(a, b)$, for any two vertices $a, b \in U\bigcup\{z\}$. By repeating this constructing inductively for every $z \in V(G)\backslash U$, the result follows.
\end{proof}

Next, we show that computing the transportation distance is a linear programming problem, with integral extreme points. To prove this we use the following result from linear programming (Theorem 2.2, Chapter 4, Yemelichev et al. \cite{yemelichev}): The polyhedron $\sP=\{\vec w: \vec b_1\leq M\vec w\leq \vec b_2, \vec d_1\leq \vec w\leq \vec d_2\}$ has integral extreme points, whenever $\vec b_1, \vec b_2, \vec d_1, \vec d_2$ are integral vectors and $M$ is {\it totally unimodular}, that is, the determinant of every sub-matrix of $M$ is 0, +1, or -1. 

\begin{lem}\label{int}
Let $G=(V(G), E(G))$ be a locally finite weighted graph. For $(x, y)\in E(G)$,  there exists $g\in \sL_1$ such that $g:V(G)\mapsto \mathbb{Z}$, $g(x)= 0$, and $g=\arg \sup_{f \in \sL_1}\{E_x(f)-E_y(f)\}$. Thus, while computing $\rc$ it suffices to optimize over integer valued 1-Lipschitz functions, and consequently $\rc$ is rational.
\end{lem}
\begin{proof}
For $(x, y)\in E(G)$ and $f \in \sL_1$, denote by $T_{xy}(f)=E_x(f)-E_y(f)=\frac{1}{d_x}\sum_{z\in N(x)}f(z)w_{xz}-\frac{1}{d_y}\sum_{z\in N(y)}f(z)w_{yz}$. Note that $T_{xy}$ is location invariant, that is, $T_{xy}(f)=T_{xy}(f+c)$, for any $c\in \R$. Therefore, w.l.o.g. generality we can assume $g(x)=0$. Thus, computing the transportation distance is equivalent to 
$$\max_f T_{xy}(f) \text{ subject to } |f(a)-f(b)|\le 1 \text{ for }(a,b)\in E(G), \text{ and } f(x)=0.$$
This is clearly a linear programming problem with $|V(G)|-1$ variables and $2|E(G)|$ constraints.

The set of all feasible functions of this linear program forms a polytope in $\R^{|V(G)|-1}$ with finitely many extreme points. Consider the digraph $D(G)$ obtained by duplicating every edge of $G$ and orienting one each in both directions. As each of the constraints of the linear program are of the form $|f(a)-f(b)|\le 1$ for $(a,b)\in E(G)$, the set of constraints can be written as $M'\vec f\leq \vec 1$ and $|\vec f|\leq \diam(G)\cdot\vec 1$, where $M$ is the incidence matrix of the digraph $D(G)$, $\vec f$ is the vector of values of $f$ with $f(x)=0$, and $\diam(G)$ is the diameter of the graph $G$. As $M$ is totally unimodular (Theorem 13.9, Schrijver \cite{schrijver}), the set of points of this linear program are integral, completing the proof of the lemma.
\end{proof}

\subsection{Reduction Lemma: Removing Large Cycles}

The transportation distance, and hence the Ricci curvature, of an edge $(x, y)\in E(G)$, is a local property depending only on vertices which are close to $x$ and $y$. In this section we prove a reduction lemma where we make the above statement precise by exactly identifying the sub-graph of $G$ which contributes to Ricci curvature of an edge $(x, y)\in E(G)$. 

Before we proceed to state the lemma, we introduce some notation which will be used throughout the paper. Consider a locally finite weighted graph $G=(V(G), E(G))$. For any two vertices $x, y \in V(G)$, such that $(x, y)\in E(G)$, we associate the following quantities (refer to Figure \ref{fig:neighborhood}(a)):
\begin{description}
\item[$\Delta_G(x, y)$] This is set of vertices in $V(G)$ which are common neighbors of both $x$ and $y$, that is, $\Delta_G(x, y)=N_G(x)\bigcap N_G(y)$. In fact, $|\Delta_G(x, y)|$ is the number of triangles in $G$ supported on $(x, y)$
\item[$P_G(x, y)$] This is the set of vertices which are at distance 2 from both $x$ and $y$. That is, $P_G(x, y)=\{v\in V(G): d_G(x, v)=d_G(y, v)=2\}$. 
\end{description}

Define $V_{(x, y)}:=N_G(x)\bigcup N_G(y)\bigcup P_G(x, y)$, and denote by $H$ the subgraph of $G$ induced by $V_{(x, y)}$.  

\begin{lem}[Reduction Lemma]
\label{truncate}
For a locally finite weighted graph $G=(V(G), E(G))$ and an edge $(x, y)\in E(G)$, $\kappa_G(x, y)=\kappa_{H}(x, y)$. Moreover, for computing $\rc$ it suffices to assume that there are no edges between $\Delta_G(x, y)$ and 
$P_G(x, y)$.
\end{lem}

\begin{figure*}[h]
\centering
\begin{minipage}[c]{0.5\textwidth}
\centering
\includegraphics[width=3.0in]
    {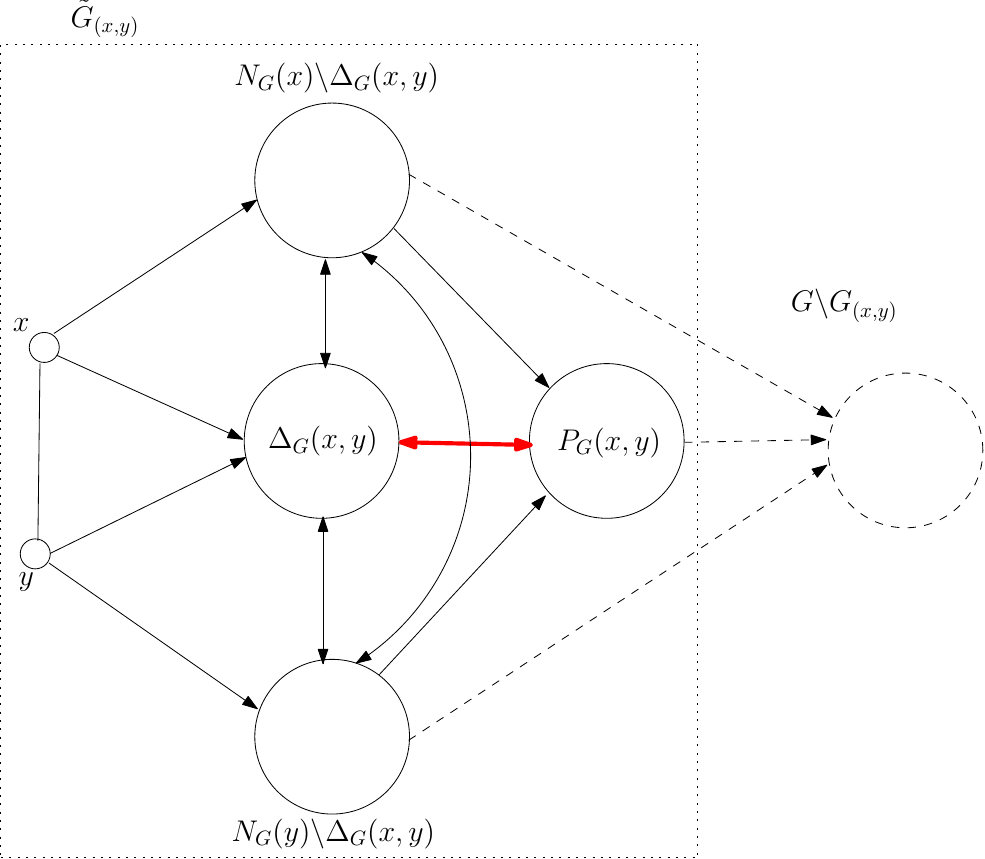}\\
\end{minipage}%
\caption{Core neighborhood of $(x, y)$: Illustration for the proof of the reduction lemma.}
\label{fig:neighborhood}
\end{figure*}

\begin{proof}
To begin with observe that $d_G(a, b)\leq d_{H}(a, b)$, for any $a, b\in V(H)$. Therefore, any function which is Lipschitz in $d_G$ is also Lipschitz in $d_{H}$. Therefore, $$W_1^{H}(m_x, m_y)\ge W_1^G(m_x, m_y).$$

To show the other inequality it suffices to show that for any 1-Lipschitz function $f:V(H)\mapsto \mathbb R$ with respect to $d_{H}$, we can define a function $g: V(G)\mapsto \mathbb R$ which is 1-Lipschitz with respect to $d_G$  and agrees with $f$ on $N_G(x, y):=N_G(x)\bigcup N_G(y)$, as the transportation distance between $m_x$ and $m_y$ only depends on the values of the function at $N_G(x, y)$. To this end, define $g=f$ on $N_G(x, y)$. Observe that if $a,b\in N_G(x, y)$, then by construction of $H$ there is a path from $a$ to $b$ of length $d_G(a,b)$ in $H$, and so $d_G(a,b)=d_{H}(a,b)$. Moreover, as $f$ is Lipschitz with respect to $d_{H}$, we have $|g(a)-g(b)|=|f(a)-f(b)|\le d_{H}(a,b)=d_G(a,b)$. Finally, applying Lemma \ref{lip} with $U=N_G(x, y)$ proves that $g$ can be extended to a Lipschitz function with respect to $d_G$ on the whole of $V(G)$. This proves that $W_1^{H}(m_x, m_y)\le W_1^G(m_x, m_y)$, and the proof of the first part of the lemma is complete.

To show the second part, it suffices to show that $\kappa_{H}(x, y)=\kappa_{H\backslash e}(x, y)$, for a edge $e$ between $\Delta_G(x, y)$ and $P_G(x, y)$. By the dual definition (\ref{rcdef2}) this means that the optimal transfer matrix $\mathcal A$ remains unchanged if we drop the edge $e$. This is equivalent to showing that for any 
$z_1\in N_G(x)$ and $z_2\in N_G(y)$ there is a shortest path connecting $z_1, z_2$ without using $e$. The following cases may arise:
\begin{description}
\item[$d_{H}(z_1,z_2)=3$] The shortest path not using $e$ in this case is $(z_1,x,y,z_2)$.

\item[$d_{H}(z_1,z_2)=2$]If $z_1\in N_G(x)\bigcap N_G(y)$ then a path of length $2$ not containing $e$ is 
$(z_1,y,z_2)$. Similarly, for $z_2\in N_G(x)\bigcap N_G(y)$ a path of length $2$ is $(z_1,x,z_2)$. Therefore, consider $z_1\in N_G(x)\backslash N_G(y)$ and  $z_2\in N_G(y)\backslash N_G(x)$. 
However, it is easy to see that as $e$ is an edge between $\Delta_G(x, y)$ and $P_G(x, y)$, a path of length $2$ containing $e$ cannot connect $z_1$ and $z_2$. Therefore, if $d_{H}(z_1,z_2)=2$, then there exists a path of length 2 not containing $e$.

\item[$d_{H}(z_1,z_2)=1$] The shortest path in this case is $(z_1,z_2)$ and $e\neq (z_1,z_2)$, as $d_{H}(z_1,x)=1, d_{H}(z_2,y)=1$. 
\end{description}
\end{proof}

If $\varphi_G(x, y)$ denotes the set of edges between $\Delta_G(x, y)$ and $P_G(x, y)$ (red edges in Figure \ref{fig:neighborhood}), we denote the {\it core neighborhood} of $(x, y)$ in $G$ as the subgraph 
$$G_{(x,y)}:=(V(H), E(H)\backslash \varphi_G(x, y)).$$
The above lemma shows that it suffices to consider only the core neighborhood subgraph $ G_{(x,y)}$ for computing the Ricci curvature of $(x, y)\in E(G)$, which greatly simplifies computations of $\rc$.


One of the few known exact formulas for Ricci curvature for graphs is the following result of Jost and Liu \cite{jost} for trees:

\begin{thm}[Jost and Liu \cite{jost}] 
For any neighboring vertices $x, y$ of a unweighted tree $T$, $$\kappa(x, y) = -2\Big(1-\frac{1}{d_x}-\frac{1}{d_y}\Big)_+.$$
\label{th:treejost}
\end{thm}

An immediate consequence of Lemma \ref{truncate} and the above theorem is the following corollary, which generalizes the formula for Ricci curvature of trees to graphs with girth $\geq 6$. This generalization was clear from the proof of Theorem \ref{th:treejost} in Jost and Lin \cite{jost}, and was also proved by Cho and Paeng \cite{chopaeng}.

\begin{cor}[\cite{chopaeng,jost}]
\label{gsix} 
For a locally finite unweighted graph $G=(V(G), E(G))$ and an edge $(x, y)\in E(G)$, 
$$\rc=-2\Big(1-\frac{1}{d_x}-\frac{1}{d_y}\Big)_+,$$ 
if there are no $3,4,\text{ and }5$ cycles supported on $(x,y)$. In particular, the above formula holds whenever the girth of 
$G$ is at least 6.
\end{cor}

\section{Bipartite Graphs and Graphs With Girth Greater Than 4}

Jost and Liu \cite{jost} obtained bounds on Ricci curvature involving the number of triangles supported on $(x, y)$, that is $|\Delta_G(x, y)|$. Their main result is stated in Theorem \ref{th:jost}, which reduces to the following when there are no triangle supported on $(x, y)$:
\begin{equation}
-2\left(1-\frac{1}{d_x}-\frac{1}{d_y}\right)_+\le \rc \le 0.
\label{eq:kappajosttrianglefree}
\end{equation}
It is also known that the lower bound is tight for the graphs with girth at least 6 \cite{chopaeng,jost}. Cho and Paeng \cite{chopaeng} also obtained bounds on the Ricci curvature of graphs with girth at least 5 in terms of the minimum degree of the graph. In this section, we obtain general upper bounds for the Ricci curvature in bipartite graphs and graphs with girth greater than 4, in terms of the structure of the core-neighborhood. (The earlier version of the article claimed exact formulas for the Ricci curvature of bipartite graphs and graphs with girth greater than 4. However, as pointed out in the recent paper of Kelly \cite{rc_reduced}, there was a bug in the arguments, which makes the lower bounds invalid. The upper bounds remain valid, and, despite the error, the proof strategy provides useful insights for computing the Ricci curvature of such graphs (cf.~\cite{rc_reduced} for more detials).)

\subsection{Ricci Curvature of Bipartite Graphs}
\label{sec:bipartite}

Recall that $N_G(x)$ and $N_G(y)$ denote the set of neighbors of $x$ and $y$, respectively. We partition $N_G(x)=N_0(x)\bigcup N_1(x)\bigcup\{y\}$, where
\begin{description}
\item[$N_1(x)$]=$\{z\in N_G(x)\backslash\{y\}: d_G(z,N_G(y)=1\}$, is the set of neighbors of $x$ which are on a 4-cycle supported on $(x, y)$.
\item[$N_0(x)$]$=N_G(x)\backslash (N_1(x)\bigcup \{y\})$, is the set of remaining neighbors of $x$, apart from $y$.
\end{description}
Similarly, we can define a partition $N_G(y)=N_0(y)\bigcup N_1(y)\bigcup\{x\}$. Now, if we assume that $G$ is bipartite, $|\Delta_G(x, y)|=0$ and $|P_G(x, y)|=0$.

\begin{figure*}[h]
\centering
\begin{minipage}[c]{1.0\textwidth}
\centering
\includegraphics[width=3.8in]
    {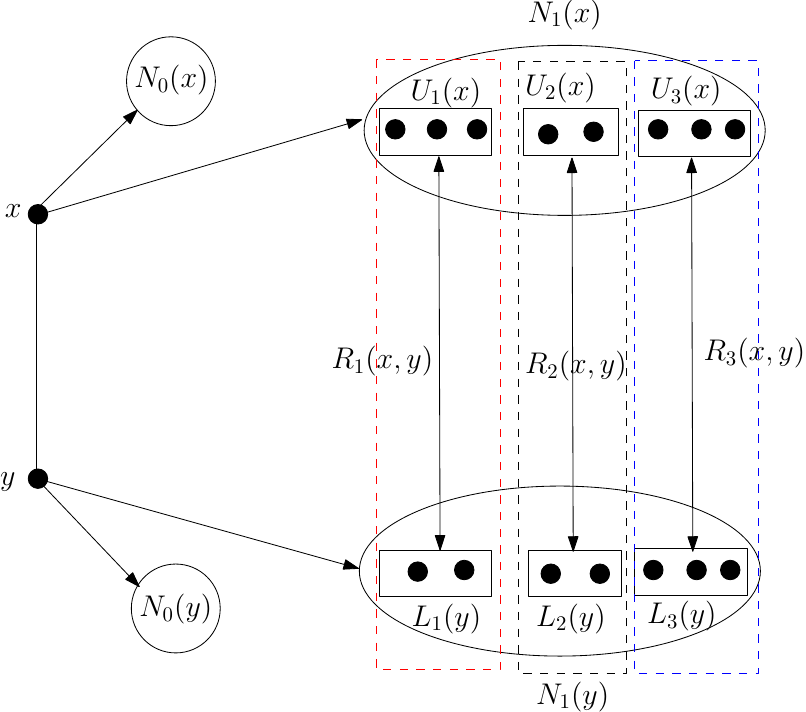}\\
\end{minipage}
\caption{Ricci curvature of bipartite graphs: Structure of the core neighborhood.}
\label{fig:rcbipartite}
\end{figure*}

\begin{thm}
\label{th:bipartite}
Let $G=(V(G), E(G))$ be a locally finite unweighted bipartite graph and $(x, y)\in E(G)$. Suppose
$R(x, y)$ is the subgraph of $G_{(x, y)}$ induced by $N_1(x)\bigcup N_1(y)$, and $R_1(x, y), R_2(x, y),$ $\ldots R_q(x, y)$ be the connected components of $R(x, y)$. If $U_a(x)=V(R_a(x, y))\bigcap N_1(x)$ and $U_a(y)=V(R_a(x, y))\bigcap N_1(y)$, for $a \in \{1, 2, \ldots, q\}$, then  
\begin{equation}
\rc \leq -2\left(1-\frac{1}{d_x}-\frac{1}{d_y}-\frac{|N_1(y)|}{d_y}+\sum_{a=1}^q \left(\frac{|U_a(y)|}{d_y}-\frac{|U_a(x)|}{d_x}\right)_+\right)_{+}.
\label{eq:thbipartite}
\end{equation}
\end{thm}

\begin{proof}
Using Lemma \ref{truncate} we can replace $G$ by its core neighborhood $G_{(x, y)}$. As $G$ is bipartite, there are no edges between $N_0(x)$ and $N_1(x)$, and between $N_0(y)$ and $N_1(y)$. Therefore, the structure of the core neighborhood $G_{(x, y)}$ is as depicted in Figure \ref{fig:rcbipartite}.

For any $Z\subseteq  V(G_{(x, y)})$ and any function $f \in \sL_1$, denote by $W_f(Z)=\sum_{z\in Z}f(z)$. Therefore, for $(x, y)\in E(G)$,
\begin{equation}
E_y(f)-E_x(f)=\frac{f(x)+W_f(N_0(y))+W_f(N_1(y))}{d_y}-\frac{f(y)+W_f(N_0(x))+W_f(N_1(x))}{d_x}.
\label{eq:wdbipartite}
\end{equation}
Lemma \ref{int} implies that it suffices to maximize $E_y(f)-E_x(f)$ over 1-Lipschitz functions  $f$ satisfying $f(x)=0$, and $f(y)\in \{-1,0,1\}$. Therefore, for $i\in \{-1, 0, 1\}$ define 
$$\kappa_i(x, y):=1-\max_{f\in \sL_1,f(y)=i}(E_y(f)-E_x(f)),$$
and observe that $\rc=\kappa_{-1}(x, y)\wedge\kappa_0(x, y)\wedge\kappa_1(x, y)$. Assuming $f(x)=0$ we consider the following three cases separately. 

\subsection*{{\it Case} 1} $f(y)=-1$. This implies that $f(z)\leq 0$ for $z\in N_G(y)$ and $f(z)\geq -1$, for $z \in N_G(x)$. Therefore, from Equation (\ref{eq:wdbipartite}) we get
\begin{equation}
E_y(f)-E_x(f)\leq\frac{1}{d_x}+\frac{|N_0(x)|+|N_1(x)|}{d_x}=1.\nonumber
\end{equation}
Moreover, this bound is attained by the function $g:V( G_{(x, y)}) \mapsto \mathbb R$:
$$g(z)=\left\{
\begin{array}{cc}
-1, & \hbox{ if } z \in N_G(x),\\
0, & \hbox{otherwise},\\
\end{array}
\right.
$$
which is 1-Lipschitz on the core neighborhood of $(x, y)$ (refer to Figure \ref{fig:rcbipartite}). This implies,
$\kappa_{-1}(x, y)=0$. 

\subsection*{{\it Case} 2} $f(y)=0$. 
Now, consider the 1-Lipschitz function $g: V( G_{(x, y)}) \mapsto \mathbb R$ 
$$g(z)=
\left\{\begin{array}{cc}
-1, & \hbox{ if } z \in N_0(x);\\
1, & \hbox{ if }  z \in N_0(y);\\
g_a(z), & \hbox{ if }  z \in U_a(x)\bigcup U_a(y);\\
0, & \hbox{ otherwise; }\\
\end{array}
\right.
$$
where for $a \in \{1, 2, \ldots, q\}$,  $$g_a(z):=\left\{
\begin{array}{cc}
-1\cdot\pmb{1}\left\{\frac{|U_a(x)|}{d_x} \geq \frac{|U_a(y)|}{d_y}\right\}, & \hbox{ if } z \in U_a(x);\\
1\cdot\pmb{1}\left\{\frac{|U_a(x)|}{d_x} < \frac{|U_a(y)|}{d_y}\right\}, & \hbox{ if } z \in U_a(y).\\
\end{array}
\right.
$$
It is easy to see that $$E_y(g)-E_y(g)=\frac{|N_0(y)|}{d_y}+\frac{|N_0(x)|+|N_1(x)|}{d_x}+\sum_{a=1}^q \left\{\frac{|U_a(y)|}{d_y}-\frac{|U_a(x)|}{d_x}\right\}\cdot\pmb{1}_{\left\{\frac{|U_a(x)|}{d_x} < \frac{|U_a(y)|}{d_y}\right\}}.$$Therefore, 
\begin{align}
\kappa(x, y) \leq \kappa_0(x, y) & \leq 1-(E_y(g)-E_x(g)) \nonumber \\ 
& = \frac{1}{d_x}-\frac{|N_0(y)|}{d_y}-\sum_{a=1}^q \left\{\frac{|U_a(y)|}{d_y}-\frac{|U_a(x)|}{d_x}\right\}\cdot\pmb{1}_{\left\{\frac{|U_a(x)|}{d_x} < \frac{|U_a(y)|}{d_y}\right\}}.
\label{eq:case2bipartite4n}
\end{align}

\subsection*{{\it Case} 3} $f(y)=1$. Consider the 1-Lipschitz function $g: V( G_{(x, y)}) \mapsto \mathbb R$ 
$$g(z)=
\left\{\begin{array}{cc}
-1, & \hbox{ if } z \in N_0(x);\\
1, & \hbox{ if }  z \in N_0(y);\\
g_a(z), & \hbox{ if }  z \in U_a(x)\bigcup U_a(y);\\
0, & \hbox{ otherwise;}\\
\end{array}
\right.
$$
where for $a \in \{1, 2, \ldots, q\}$,  $$g_a(z):=\left\{
\begin{array}{cc}
-1\cdot\pmb{1}_{\left\{\frac{|U_a(x)|}{d_x} \geq \frac{|U_a(y)|}{d_y}\right\}}+1\cdot\pmb{1}_{\left\{\frac{|U_a(x)|}{d_x} < \frac{|U_a(y)|}{d_y}\right\}}, & \hbox{ if } z \in U_a(x);\\
2\cdot\pmb{1}_{\left\{\frac{|U_a(x)|}{d_x} < \frac{|U_a(y)|}{d_y}\right\}}, & \hbox{ if } z \in U_a(y).\\
\end{array}
\right.
$$
This implies,   
\begin{align}
\kappa(x, y)  \leq \kappa_1(x, y) & \leq 1-(E_y(g)-E_x(g)) \nonumber \\ 
&=\frac{2}{d_x}-\frac{2|N_0(y)|}{d_y}-2\sum_{a=1}^q \left\{\frac{|U_a(y)|}{d_y}-\frac{|U_a(x)|}{d_x}\right\}\cdot\pmb{1}_{\left\{\frac{|U_a(x)|}{d_x} < \frac{|U_a(y)|}{d_y}\right\}}.
\label{eq:case3bipartite4}
\end{align}

Note that the RHS of \eqref{eq:case3bipartite4} is twice the RHS of \eqref{eq:case2bipartite4n}. Therefore, taking the minimum of the RHS of \eqref{eq:case3bipartite4} with zero and noting that $|N_0(y)|=d_y - |N_1(y)| -1$, the upper bound in \eqref{eq:thbipartite} follows. 
\end{proof}


\subsection{Ricci Curvature of Graphs with Girth Greater Than 4}
\label{sec:gfive}

In this section we consider graphs $G$ with girth greater than 4. As before, partition $N_G(x)=N_0(x)\bigcup N_2(x)\bigcup \{y\}$, where
\begin{description}
\item[$N_2(x)$]$=\{z\in N_G(x)\backslash\{y\}: d_G(z,N_G(y)=2\}$. 
\item[$N_0(x)$]$=N_G(x)\backslash(N_2(x)\bigcup \{y\})$, is the set of remaining neighbors of $x$, apart from $y$.
\end{description}
Similarly, we can define a partition $N_G(y)=N_0(y)\bigcup N_2(y)\bigcup\{x\}$.


\begin{figure*}[h]
\centering
\begin{minipage}[c]{1.0\textwidth}
\centering
\includegraphics[width=3.8in]
    {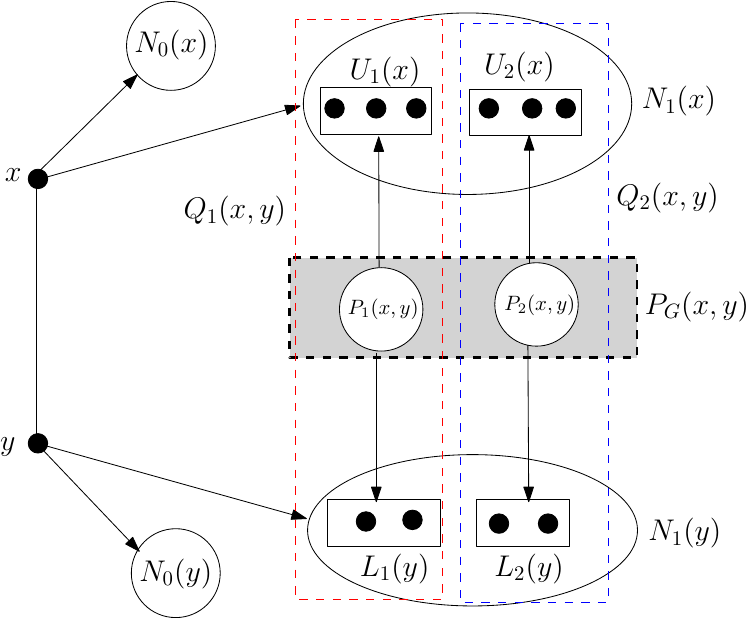}\\
\end{minipage}
\caption{Ricci curvature of graphs with girth at least 5: Structure of the core neighborhood.}
\label{fig:gfive}
\end{figure*}

\begin{thm}
\label{th:gfive}
Let $G=(V(G), E(G))$ be a locally finite unweighted graph with girth $g(G)\ge 5$ and $(x, y)\in E(G)$. Suppose
$Q(x, y)$ is the subgraph of $G_{(x, y)}$ induced by $N_2(x)\bigcup N_2(y)\bigcup P_G(x, y)$, and $Q_1(x, y),$ $ Q_2(x, y), \ldots Q_q(x, y)$ be the connected components of $Q(x, y)$. If $L_a(x)=V(Q_a(x, y))\bigcap N_2(x)$ and $L_a(y)=V(Q_a(x, y))\bigcap N_2(y)$, for $a \in \{1, 2, \ldots, q\}$, then $\rc \leq a(x, y)\wedge b(x, y)$, where
\begin{equation}
a(x, y)=-\left(1-\frac{1}{d_x}-\frac{1}{d_y}\right)_+.
\label{eq:thgfiven}
\end{equation}
and
\begin{equation}
b(x, y)=-2\left(1-\frac{1}{d_x}-\frac{1}{d_y}-\frac{|N_2(x)|}{2d_x}+\frac{1}{2}\sum_{a=1}^q\left(\frac{|L_a(x)|}{d_x}-\frac{|L_a(y)|}{d_y}\right)_+\right)_+.
\label{eq:thgfive}
\end{equation}
\end{thm}

\begin{proof}
The proof of this theorem is similar to the proof of Theorem \ref{th:bipartite}. Using Lemma \ref{truncate} we can replace $G$ by its core neighborhood $G_{(x, y)}$, which is depicted in Figure \ref{fig:gfive}. For any $Z\subseteq  V(G_{(x, y)})$ and any function $f \in \sL_1$, denote by $W_f(Z)=\sum_{z\in Z}f(z)$. Therefore, for $(x, y)\in E(G)$,
\begin{equation}
E_y(f)-E_x(f)=\frac{f(x)+W_f(N_0(y))+W_f(N_2(y))}{d_y}-\frac{f(y)+W_f(N_0(x))+W_f(N_2(x))}{d_x}
\label{eq:wdgfive}
\end{equation}
As before, for $i\in \{-1, 0, 1\}$ define 
$$\kappa_i(x, y):=1-\max_{f\in \sL_1,f(y)=i}(E_y(f)-E_x(f)),$$
and observe that $\rc=\kappa_{-1}(x, y)\wedge\kappa_0(x, y)\wedge\kappa_1(x, y)$. Assuming $f(x)=0$ we consider the following three cases separately.

\subsection*{{\it Case} 1} $f(y)=-1$. This implies that $f(z)\leq 0$ for $z\in N_G(y)$ and $f(z)\geq -1$, for $z \in N_G(x)$. Therefore, from Equation (\ref{eq:wdgfive}) we get
\begin{equation}
E_y(f)-E_x(f)\leq\frac{1}{d_x}+\frac{|N_0(x)|+|N_2(x)|}{d_x}=1.\nonumber
\end{equation}
Moreover, this bound is attained by the function $g(z):=-1\cdot \pmb 1_{\{z\in N_G(x)\bigcup\{y\}\}}$ which is 1-Lipschitz on the core neighborhood of $(x, y)$. This implies,
$\kappa_{-1}(x, y)=0$. 

\subsection*{{\it Case} 2} $f(y)=0$. This implies that $f(z)\leq 1$ for $z\in N_G(y)$ and $f(z)\geq -1$, for $z \in N_G(x)$. Therefore, from Equation (\ref{eq:wdgfive}) we get,
\begin{eqnarray}
E_y(f)-E_x(f)&\leq&\frac{|N_0(y)|+|N_2(y)|}{d_y}+\frac{|N_0(x)|+|N_2(x)|}{d_x},\nonumber\\
&=&2-\frac{1}{d_x}-\frac{1}{d_y}.
\label{eq:case2gfive}
\end{eqnarray}
This implies that $\kappa_0(x, y)\geq -1+\frac{1}{d_x}+\frac{1}{d_y}$. 

Now, consider the 1-Lipschitz function $g: V( G_{(x, y)}) \mapsto \mathbb R$ 
$$g(z)=
\left\{\begin{array}{cc}
-1, & \hbox{ if } z \in N_0(x)\bigcup N_2(x);\\
1, & \hbox{ if }  z \in N_0(y)\bigcup N_2(y);\\
0, & \hbox{ otherwise. }\\
\end{array}
\right.
$$
It is easy to see that $E_y(g)-E_y(g)=2-\frac{1}{d_x}-\frac{1}{d_y}.$ Therefore, $g$ attains the lower bound in Equation (\ref{eq:case2gfive}) and we have
\begin{eqnarray}
\kappa_0(x, y)&=&-1+\frac{1}{d_x}+\frac{1}{d_y}.
\label{eq:case2gfive_N}
\end{eqnarray} 

\subsection*{{\it Case} 3} $f(y)=1$. Now, consider the 1-Lipschitz function $g: V( G_{(x, y)}) \mapsto \mathbb R$ with $g(x)=0$, $g(y)=1$ and 
$$g(z)=
\left\{\begin{array}{cc}
-1, & \hbox{ if } z \in N_0(x);\\
2, & \hbox{ if }  z \in N_0(y);\\
g_a(z), & \hbox{ if }  z \in L_a(x)\bigcup L_a(y).
\end{array}
\right.,
$$
where for $a \in \{1, 2, \ldots, q\}$,  $$g_a(z):=\left\{
\begin{array}{cc}
-1\cdot\pmb{1}_{\left\{\frac{|L_a(x)|}{d_x} \geq \frac{|L_a(y)|}{d_y}\right\}}, & \hbox{ if } z \in L_a(x);\\
1\cdot\pmb{1}_{\left\{\frac{|L_a(x)|}{d_x} \geq \frac{|L_a(y)|}{d_y}\right\}}+2\cdot\pmb{1}_{\left\{\frac{|L_a(x)|}{d_x} < \frac{|L_a(y)|}{d_y}\right\}}, & \hbox{ if } z \in L_a(y);\\
1\cdot\pmb{1}_{\left\{\frac{|L_a(x)|}{d_x} < \frac{|L_a(y)|}{d_y}\right\}}, & \hbox{ if } z \in P_G(x, y).\\
\end{array}
\right.
$$
This implies,   
\begin{align}
\kappa(x, y)  \leq \kappa_1(x, y) & \leq 1-(E_y(g)-E_x(g)) \nonumber \\ 
&=-2+\frac{2}{d_x}+\frac{2}{d_y}+\frac{|N_2(x)|}{d_x}-\sum_{a=1}^q\left\{\frac{|L_a(x)|}{d_x}-\frac{|L_a(y)|}{d_y}\right\}\cdot\pmb{1}_{\left\{\frac{|L_a(x)|}{d_x} \geq \frac{|L_a(y)|}{d_y}\right\}}.
\label{eq:case3five4}
\end{align}

Finally, combining \eqref{eq:case2gfive_N} and \eqref{eq:case3five4} and recalling the definition of $a(x, y)$ and $b(x, y)$ from \eqref{eq:thgfiven} and \eqref{eq:thgfive}, the result follows. 
\end{proof}

\section{Ricci-Flat Graphs of Girth 5} 

The notion of Ricci-flat graphs was first introduced by Chung and Yau \cite{chungyau}. It was defined for regular graphs and was used to prove logarithmic Harnack inequalities for graphs. These inequalities are useful in bounding  the Log-Sobolev constant, and, consequently, the mixing times of Markov chains \cite{diaconis_saloff-coste}. Recently, Lin et al. \cite{ricciflat,tohoku} defined Ricci-flat graphs as graphs where the modified Ricci curvature vanishes on every edge. This definition is motivated by the fact that Ricci-flat manifolds are Riemannian manifolds where the Ricci curvature vanishes. Moreover, this definition does not require the graph to be regular. Ricci flat graphs defined by Chung and Yau are not necessarily Ricci-flat in this sense (for constructions of Ricci-flat graphs under both definitions refer to Lin et al. \cite{ricciflat}).

In this paper, we consider Ricci-flat graphs with respect to Olliver's coarse Ricci curvature, that is, a graph $G$ is said to be {\it Ricci-flat} if $\rc=0$, for all $(x, y)\in E(G)$. The following corollary characterizes Ricci-flat graphs with girth at least 5. Hereafter, only locally finite simple unweighted graphs are considered.

\begin{cor}
A connected graph $G$ is a Ricci-flat graph with $g(G)\geq 5$, if and only if $G$ is one of the following: the path $P_n~(n \geq 2)$, the infinite ray, the infinite path, the cycle $C_n~(n\geq 5)$, or the star graph $T_n~(n\geq 3)$.
\label{cor:ricciflatfive}
\end{cor}

\begin{proof} Suppose $\rc=0$ for all $(x, y)\in E(G)$ and $g(G)\geq 5$. Then it follows from Theorem \ref{th:gfive} that $\frac{1}{d_x}+\frac{1}{d_y}\ge 1$, for all $(x, y)\in E(G)$. This implies that either $d_x=d_y=2$, or $d_x\wedge d_y=1$, for all  $(x, y)\in E(G)$.The following two cases arise: 

\begin{description} 
\item[{\it Case} 1] There is a vertex $v\in V(G)$ with $d_v\geq 3$.  Then all neighbors of $v$ must have degree $1$, with no edges left to connect to other vertices. Thus, the graph $G$ must be a $n$-star $T_n$ rooted at $v$. 

\item[{\it Case} 2]$d_x\le 2$ for all $x\in V(G)$.  If $d_x=2$, for all $x\in V(G)$, then it is easy to see that $G$ is an infinite path or a cycle of length at least 5. Hence, it suffices to assume that there is at least one vertex $V(G)$ with degree $1$. Then there can be at most 2 vertices with degree $1$. Thus, depending on whether the number of degree 1 vertices is one or two, $G$ is either the infinite half ray or  the finite path $P_n$, respectively.
\end{description}
\end{proof}

Lin et al. \cite{ricciflat} characterized Ricci-flat graphs of girth 5 using their modified definition of Ricci curvature. The above corollary is the analogous version of their result using Ollivier's original definition of coarse Ricci curvature. As it happens, the structure of Ollivier's Ricci-flat graphs of girth 5 is much simpler than the structure of the modified Ricci-flat graphs.  Apart from the infinite path and the cycle, the modified Ricci-flat graphs of girth 5 include quite complicated graphs, such as the Peterson graph, the dodecahedral graph, and the half-dodecahedral graph. Ollivier's Ricci-flat graph however includes the $n$-star, which is not included in the modified definition. This illustrates that the structure of Ricci-flat graphs in the two definitions of Ricci curvature are, in fact, quite different.

\section{A General Lower Bound With Maximum Matching and Ricci-Flat Regular Graphs of Girth 4}
\label{sec:matching}


In this section we establish connections between Ricci curvature and the size of the matching in the core neighborhood subgraph. We prove a general lower bound on the Ricci curvature in terms of the size of the maximum matching. The bound is often nearly tight, especially in triangle-free regular graphs. Using this bound we prove a necessary and sufficient condition on the structure of Ricci-flat regular graphs of girth 4.

\subsection{A General Lower Bound With Maximum Matching} Let $G=(V, E)$ be a fixed graph, with $(x, y)\in E(G)$. Recall that $N_G(x)$ and $N_G(y)$ denote the set of neighbors of $x$ and $y$, and $\Delta_G(x, y)=N_G(x) \bigcap N_G(y)$. Define,  $Q_G(x)=N_G(x)\backslash \Delta_G(x, y)$ and $Q_G(y)=N_G(y)\backslash \Delta_G(x, y)$. Let $\overline{H}_G(x, y)$ be the subgraph of $G$ induced by the vertices in $Q_G(x)\bigcup Q_G(y)$. A {\it matching} in $\overline{H}_G(x, y)$ is a collection of disjoint edges $(a, b)\in E(G)$, with $a\in Q_G(x)$ and $b\in Q_G(y)$.

\begin{thm}
Let $G=(V, E)$ be a fixed graph, with $(x, y)\in E(G)$. If $|M_G(x, y)|~(\leq |Q_G(x)|\wedge |Q_G(y)|)$ is the size of the maximum matching in $\overline{H}_G(x, y)$, then 
$$\rc\ge \frac{|\Delta_G(x, y)|}{d_x\vee d_y}-2\Big(1-\frac{|M_G(x, y)|+|\Delta_G(x, y)|}{d_x\vee d_y}\Big).$$
Moreover, if $|M_G(x, y)|=|Q_G(x)|\wedge |Q_G(y)|$, then $\rc\ge \frac{|\Delta_G(x, y)|}{d_x\vee d_y}-2\Big(1-\frac{d_x\wedge d_y}{d_x\vee d_y}\Big)$.
\label{th:matching}
\end{thm}

\begin{proof}
W.l.o.g. assume $d_x\leq d_y$ and let $|M_G(x, y)|=k$. By the definition of matching there exists $T_x:=\{a_1, a_2, \ldots, a_k\}\subseteq Q_G(x)$ and $T_y:=\{b_1, b_2, \ldots, b_k\}\subseteq Q_G(y)$ such that $d_G(a_i, b_i)=1$. Consider any 1-Lipschitz function $f: V(G)\mapsto \mathbb Z$. Lemma \ref{int} implies that it suffices to optimize over $f(x)=0$. This means that $|f(z)|\le 1$ for all $z\in N_G(x)$, and $|f(a_i)-f(b_i)|\le 1$. 
\begin{eqnarray}
T_{xy}(f)&=&\sum\limits_{i=1}^k\left(\frac{f(b_i)}{d_y}-\frac{f(a_i)}{d_x}\right)+\sum\limits_{z\in \Delta_G(x, y)}\left(\frac{f(z)}{d_y}-\frac{f(z)}{d_x}\right)+\sum\limits_{z\in Q_G(y)\backslash T_y}\frac{f(z)}{d_y}-\sum\limits_{z\in Q_G(x)\backslash T_x}\frac{f(z)}{d_x}\nonumber\\
&\le &\frac{k}{d_y}+ \left(\frac{1}{d_x}-\frac{1}{d_y}\right)\left\{\sum\limits_{i=1}^k|f(a_i)|+\sum\limits_{z\in \Delta_G(x, y)}|f(z)|\right\}+\sum\limits_{z\in Q_G(y)\backslash T_y}\frac{f(z)}{d_y}+\frac{|Q_G(x)\backslash T_x|}{d_x}\nonumber\\
&\le &\frac{k}{d_y}+ \left(\frac{1}{d_x}-\frac{1}{d_y}\right)\left\{k+|\Delta_G(x, y)|\right\}+\frac{2|Q_G(y)\backslash T_y|}{d_y}+\frac{|Q_G(x)\backslash T_x|}{d_x}\nonumber\\
&=&\frac{k}{d_y}+ \left(\frac{1}{d_x}-\frac{1}{d_y}\right)\left\{k+|\Delta_G(x, y)|\right\}+\frac{2(|Q_G(y)|-k)}{d_y}+\frac{|Q_G(x)|-k}{d_x}
\label{eq:matching}
\end{eqnarray}
The result follows from noting that $\rc=1-\sup_{f\in \sL_1, f(x)=0}T_{xy}(f)$, and $|Q_G(x)|=d_x-|\Delta_G(x, y)|$, $|Q_G(y)|=d_y-|\Delta_G(x, y)|$. If $d_x\le d_y$ and $k=|Q_G(x)|\wedge |Q_G(y)|$, then $k=|Q_G(x)|=d_x-|\Delta_G(x, y)|$, and the result follows from the previous bound by direct substitution.
\end{proof}

Combining the lower bound in the above theorem and the upper bound from Theorem \ref{th:jost} we get
\begin{equation}
\frac{|\Delta_G(x, y)|}{d_x\vee d_y}-2\Big(1-\frac{|M_G(x, y)|+|\Delta_G(x, y)|}{d_x\vee d_y}\Big)\leq \rc\leq \frac{|\Delta_G(x, y)|}{d_x\vee d_y}.
\label{eq:matchingbound}
\end{equation}

In the following section we obtain a necessary and sufficient condition on the structure of regular graphs for which the upper bound in Equation \ref{eq:matchingbound} is tight. As a consequence we characterize Ricci-flat regular graphs with girth 4.

\subsection{Ricci-Flat Regular Graphs with Girth 4} 
Recall that the {\it Birkhoff} polytope $B_n$ is the convex polytope in $\mathbb R^{n^2}$ whose points are the {\it doubly stochastic matrices}, that is, the $n \times n$ matrices whose entries are non-negative real numbers and whose rows and columns each add up to 1 \cite{ziegler}. The {\it Birkhoff-von Neumann theorem} states that the extreme points of the Birkhoff polytope are the {\it permutation matrices}, that is, matrices with exactly one entry 1 in each row and each column and 0 elsewhere.

Using this result on the Birkhoff polytope and Theorem \ref{th:matching} it is easy to get a necessary and sufficient condition on the structure of Ricci-flat graphs with girth 4 which are regular. In fact, we shall prove a much stronger result where we give a necessary and sufficient condition on the structure of graphs for which the upper bound in Equation \ref{eq:matchingbound} is an equality.

\begin{thm}
For be a graph $G=(V, E)$, with $(x, y)\in E(G)$ and $d_x=d_y=d$, $\rc=\frac{|\Delta_G(x, y)|}{d}$, if and only if there is a perfect matching between $Q_G(x)$ and $Q_G(y)$. 
\end{thm}

\begin{proof}If there is a perfect matching between $Q_G(x)$ and $Q_G(y)$, then by Theorem \ref{th:matching}, $\rc= \frac{\Delta_G(x, y)}{d}$.

Conversely, suppose $\rc=\frac{\Delta_G(x, y)}{d}$. From Equation \ref{rcdef2}
$$\rc=1-\inf_{\nu\in \mathcal{A}}\sum_{z_1\in N_G(x)}\sum_{z_2\in N_G(y)}\nu(z_1,z_2)d(z_1,z_2),
$$
where $\mathcal{A}$ is the set of all $d\times d$ matrices with entries indexed by $N_G(x)\times N_G(y)$ such that $\nu(x',y')\ge 0$, $\sum_{z\in N_G(y)}\nu(x',z)=\frac{1}{d}$, and $\sum_{z\in N_G(x)}\nu(z,y')=\frac{1}{d}$, for all $x'\in N_G(x)$ and $y'\in N_G(y)$. Therefore, $\mathcal A$ forms a Birkhoff polytope in $\mathbb R^{d^2}$ (after multiplying with $d$), and since $\rc$ is a linear function defined over $\mathcal A$, it is maximized at one the extreme points. Therefore, by the Birkhoff-von Neumann theorem the optimal transfer plan is a permutation matrix.

To complete the proof, note that the optimal transfer plan, which is given by a permutation matrix, cannot transfer any mass to or from $\Delta_G(x,y)$. Moreover, each transfer must be over a path of length $1$. This is because a mass of $1/d$ needs to be transferred by a path of length at least $1$ for all vertices in $N_G(x)\backslash\Delta_G(x,y)$, which already gives $W_1(x,y)\ge 1-\Delta_G(x,y)/d$, and so any further mass transfer will result in $\kappa(x,y)<\Delta_G(x,y)/d$, a contradiction.  This implies that there must be a perfect matching between $Q_G(x)$ and $Q_G(y)$.
\end{proof}

When $G$ is triangle free, by definition $Q_G(x)=N_G(x)$ and $Q_G(y)=N_G(y)$. The following corollary is then immediate from the above theorem:

\begin{cor}
A connected graph $G$ with $g(G)=4$ and $d_x=d_y=d$ has $\rc=0$ if and only if there is a perfect matching between $N_G(x)$ and $N_G(y)$.
\label{cor:ricciflat4}
\end{cor}

This implies that regular, triangle-free Ricci-flat graphs must have a perfect matching between $N_G(x)$ and $N_G(y)$ for all $(x, y)\in E(G)$.  The $n$-dimensional integer lattice $\mathbb Z^n$, the $n$-dimensional hypercube $C_2^n$, a cycle $C_n$ of length $n\ge 4$, and the complete bipartite graph $K_{n, n}$, are examples of regular Ricci-flat graphs of girth 4. Identifying the set of all such regular graphs appears to be difficult graph theory problem, which remains open.

\section{Ricci Curvature of Random Graphs}

In this section we study the behavior of Ollivier's Ricci-curavture for Erd\H os-R\'enyi random graphs 
$\mathbb G(n ,p)$ in different regimes of $p$.  As we saw in the previous section, the Ricci curvature is greatly determined by the size of matchings in the core neighborhood subgraph. In this section we will prove a technical matching lemma, establish properties of matchings in random bipartite graphs, and use these results to obtain Ricci curvature of random graphs.

\subsection{A More Technical Matching Lemma}Let $G=(V, E)$ be a fixed graph, with $(x, y)\in E(G)$. Define, $R_G(x)=(N_G(x)\backslash\{y\})\backslash \Delta_G(x, y)$ and $R_G(y)=(N_G(y)\backslash\{x\})\backslash \Delta_G(x, y)$. Let $H_G(x, y)$ be the subgraph of $G$ induced by the vertices in $R_G(x)\bigcup R_G(y)$. The subgraph $H_G(x, y)$ is said to have a $m$-{\it matching} of {\it size} $k$ if there exists $T_x:=\{a_1, a_2, \ldots, a_k\}\subseteq R_G(x)$ and $T_y:=\{b_1, b_2, \ldots, b_k\}\subseteq R_G(y)$ such that $d_G(a_i, b_i)\leq m$, for $m \in \mathbb Z^+$ and $k \leq |R_G(x)|\wedge |R_G(y)|$.  Note that a $1$-{\it matching} of size $k$ is just the standard bipartite matching of size $k$ between $R_G(x)$ and $R_G(y)$ in the subgraph $H_G(x, y)$.

The following lemma gives a lower bound on the Ricci curvature in terms of the size of 2-matchings in $H_G(x, y)$. We shall use this lemma later to bound Ricci curvature of random graphs.

\begin{lem}\label{bp-use}Let $G=(V, E)$ be a fixed graph, with $(x, y)\in E(G)$, and $H_G(x, y)$ be the subgraph of $G$ induced by the vertices in $R_G(x)\bigcup R_G(y)$. If there exists a 2-matching of size $k~(\leq |R_G(x)|\wedge |R_G(y)|)$ in the  $H_G(x, y)$,
then
$$\rc\ge -2+\frac{3|\Delta_G(x, y)|+k+2}{d_x\vee d_y}.$$
Moreover, if $k=|R_G(x)|\wedge |R_G(y)|$, then $\rc\ge -2+\frac{2|\Delta_G(x, y)|+d_x\wedge d_y+1}{d_x\vee d_y}$. 
\end{lem}

\begin{proof}
W.l.o.g. assume $d_x\leq d_y$. By the definition of $m$-matching there exists $T_x:=\{a_1, a_2, \ldots, a_k\}\subseteq R_G(x)$ and $T_y:=\{b_1, b_2, \ldots, b_k\}\subseteq R_G(y)$ such that $d_G(a_i, b_i)\leq 2$. Consider any 1-Lipschitz function $f: V(G)\mapsto \mathbb Z$. Lemma \ref{int} implies that it suffices to optimize over $f(x)=0$. This means that $|f(z)|\le 1$ for all $z\in N_G(x)$, and $|f(a_i)-f(b_i)|\le 2$. If $T_{xy}(f)=\mathbb E_x(f)-\mathbb E_y(f)$, then from calculations similar to the proof of Theorem \ref{th:matching}
\begin{eqnarray}
T_{xy}(f)&=&-\frac{f(y)}{d_x}+\sum\limits_{i=1}^k\left(\frac{f(b_i)}{d_y}-\frac{f(a_i)}{d_x}\right)+\sum\limits_{z\in \Delta_G(x, y)}\left(\frac{f(z)}{d_y}-\frac{f(z)}{d_x}\right)+\sum\limits_{z\in R_G(y)\backslash T_y}\frac{f(z)}{d_y}-\sum\limits_{z\in R_G(x)\backslash T_x}\frac{f(z)}{d_x}\nonumber\\
&\leq & \frac{|f(y)|}{d_x}+\frac{2k}{d_y}+ \left(\frac{1}{d_x}-\frac{1}{d_y}\right)\left\{k+|\Delta_G(x, y)|\right\}+\frac{2(|R_G(y)|-k)}{d_y}+\frac{|R_G(x)|-k}{d_x}.
\label{eq:matching}
\end{eqnarray}
Note that $\rc=1-\sup_{f\in \sL_1, f(x)=0}T_{xy}(f)$. Therefore, simplifying Equation \ref{eq:matching} the result follows. Moreover, if $d_x\le d_y$ and $k=|R_G(x)|\wedge |R_G(y)|$, then $k=|R_G(x)|=d_x-1-|\Delta_G(x, y)|$, and the result follows from the previous bound by direct substitution.
\end{proof}

\subsection{Matchings in Random Bipartite Graphs}
Matchings in random graphs are well studied in the literature \cite{janson}, beginning with the celebrated result of Erd\H os and R\'enyi \cite{erdosrenyi} on the existence of perfect matchings. 
The proof of this result relies on the Hall's marriage theorem. We shall use a stronger version of the Hall's theorem to obtain an analogous result about the existence of near-perfect matching in random bipartite graphs. We will use this result later to prove Ricci curvature of random graphs.


Recall Hall's marriage theorem which states that a bipartite graph $G= (V, E)$ with bipartition $(A, B)$ has a perfect matching if and only if for all $X\subseteq A$, $|N_G(X)|\geq |X|$, where $N_G(X)=\bigcup_{x\in X}N_G(x)$. We shall need the following strengthening of the Hall's theorem \cite{diestel}:

\begin{thm}\cite{diestel}
\label{mhall}
Consider a bipartite graph $G= (V, E)$ with bipartition $(A, B)$. For $X\subseteq A$, define $\delta(X)=|X|-|N_G(X)|$, where $N_G(X)=\bigcup_{x\in X}N_G(x)$. Let
$\delta_{\max}= \max_{X\subseteq A}\delta(X)$. Then the size of the maximum matching in $G$ is 
$|A|-\delta_{\max}$. \hfill $\Box$
\end{thm}

Using this theorem we now prove the following lemma about the existence of near-perfect matching in random bipartite graphs.

\begin{lem} 
\label{matched}For every $\varepsilon\in (0,1)$,
$$\lim_{n \rightarrow \infty}\mathbb P(\mathbb G(n, n, p)~has~a~matching~of~size~n(1-\varepsilon))=
\left\{
\begin{array}{ccc}
 0& if~np\rightarrow0, \\
 1& if~np \rightarrow\infty.
\end{array}
\right.
.$$
\end{lem}

\begin{proof}
Let $G\sim\mathbb G(n, n, p)$ be a random bipartite graph with bipartition $(A, B)$, with $|A|=|B|=n$ and edge probability $p$. 

If $np\rightarrow 0$ and there is a matching of size $n(1-\varepsilon)$ in $G$, then $|E(G)| \geq n(1-\varepsilon)$. But $E(G)\sim Bin(n^2,p)$, and so $\mathbb P(E(G)\ge n(1-\varepsilon))\rightarrow 0$ by Markov's inequality. 

Next, suppose $np\rightarrow\infty$. Let $X\subseteq A$, with $|X|\geq \varepsilon n$. If $D(X)=\sum_{x\in X}d_x$, then $D(X)\sim Bin(n|X|,p)$, and by Hoeffding's inequality
$\mathbb P(D(X)<|X|)\le \exp(-2|X|(np-1)^2)$. Thus, by a union bound, 
\begin{eqnarray}
\mathbb P(\exists X\subseteq A: |N_G(X)| < |X|, |X|\geq \varepsilon n)&\leq &\mathbb P(\exists X\subseteq A: D(X) < |X|, |X|\geq \varepsilon n)\nonumber\\
&\leq &\sum\limits_{|X|=n\varepsilon}^n{n\choose |X|}\exp(-2|X|(np-1)^2)\nonumber\\
&\le & 2^n\exp(-2\varepsilon n(np-1)^2)\rightarrow 0.\nonumber
\end{eqnarray}
Thus, with probability $1-o(1)$, for any $X\subseteq A$, with $|X|\geq \varepsilon n$,  $|N_G(X)|\geq |X|$. Therefore, from Theorem \ref{mhall}, $G$ has a matching of size $n(1-\varepsilon)$ with probability $1-o(1)$, when $np\rightarrow\infty$. 
\end{proof}

\begin{remark}\label{rematched}
Lemma \ref{matched} immediately gives that for a random bipartite graph  $\G(m,n,p)$ with $(m\wedge n)p\rightarrow \infty$,  $\lim_{m, n \rightarrow \infty}\mathbb P(\mathbb G(m, n, p)~has~a~matching~of~size~(m\wedge n)(1-\epsilon))=1$.
\end{remark}

\subsection{Ricci Curvature of Random Bipartite Graphs} 

We are now ready to state and prove our result on the Ricci curvature of random bipartite graphs. Let $\mathbb G(n, n, p)$ be a random bipartite graph with bipartition $(A_n, B_n)$. 
Let $a \in A_n$ and $b\in B_n$ be two fixed vertices. For $G\sim \mathbb G(n, n, p)$, conditioned
on the edge $(a, b)$ being present, denote by $\kappa_n(a, b)$ the Ricci curvature of the edge $(a, b)$ in $G$.





\begin{thm}\label{bp}
Let $\mathbb G(n, n, p_n)$ be the distribution of a random bipartite graph with bipartition $(A_n, B_n)$, conditioned on the edge $(a, b)$ being present. 
\begin{enumerate}[(a)]
\item
If $np_n\rightarrow 0$ then  with probability $1-o(1)$ the edge $(a,b)$ is isolated, and consequently 
$\kappa_n(a, b)=0$.

\item
If $np_n\rightarrow \lambda$ for $0<\lambda<\infty$, then $\kappa_n(a,b)\stackrel{\sD}{\rightarrow}-2\left(1-\frac{1}{1+X_1}-\frac{1}{1+X_2}\right)_+$,
where $X_1, X_2$ are independent $Poisson(\lambda)$ random variables, that is, the Ricci curvature converges in distribution to Ricci curvature of its limiting tree.

\item
If $np_n\rightarrow\infty$, and $np_n^2\rightarrow 0$ then $\kappa_n(a,b)\stackrel{\sP}{\rightarrow}-2$.

\item
If $np_n^2\rightarrow\infty$, then $\kappa_n(a, b)\stackrel{\sP}{\rightarrow}0$.

\end{enumerate}
\end{thm}

\begin{proof}
Let $G\sim \mathbb G(n, n, p_n)$ be a random bipartite graph with bipartition $A_n=\{\alpha_1, \alpha_2,\ldots, \alpha_n\}$ and $B_n=\{\beta_1, \beta_2, \ldots, \beta_n\}$. For $i,j\in \{1, 2, \ldots, n\}:=[n]$ define let $\delta_{ij}=1$ if $(\alpha_i, \beta_j)\in E(G)$ and $0$, otherwise. 
Define $X^a_n:=\sum_{j\in [n]\backslash\{b\}}\delta_{aj}$ and $X_n^b:=\sum_{i\in [n]\backslash\{a\}}\delta_{ib}$. Clearly, $(X^a_n, X^b_n)$ are independent binomial random variables with parameters $(n-1,p_n)$, which are measurable with respect to $\sF_n:=\sigma(\delta_{aj},\delta_{ib }, i\in [n]\backslash\{a\}, j\in [n]\backslash\{b\})$.
\begin{enumerate}[(a)]
\item
In this case, $\mathbb{P}(\{X^a_n\ne 0\}\bigcup \{X^b_n\ne 0\})\le \mathbb{E}(X^a_n+X^b_n)\le 2np_n\rightarrow 0$. Therefore, with probability $1-o(1)$ the edge $(a, b)$ is isolated. 

\item
Observe that, in this case, 
\begin{align*}
\mathbb P\left(\sum_{i\in [n]\backslash\{b\}}\sum_{j\in [n]\backslash\{a\}}\delta_{aj}\delta_{ji}\delta_{ib}\ne 0\right)\leq \sum_{i\in [n]\backslash\{b\}}\sum_{j\in [n]\backslash\{a\}} \mathbb E(\delta_{aj}\delta_{ji}\delta_{ib})\le n^2p_n^3=O(1/n).
\end{align*}
Therefore, with probability $1-O(1/n)$ there are no 4-cycles in $G$ supported on $(a,b)$. As $G$ is bipartite, the girth of $G$ is at least 6, and Lemma \ref{gsix} implies that  $$\kappa_n(a,b)=-2\left(1-\frac{1}{1+X_n^a}-\frac{1}{1+X_n^b}\right)_+.$$ The conclusion follows on noting that $(X_n^a,X_n^b)\stackrel{\sD}\rightarrow (X_1, X_2)$, where $X_1, X_2$ are independent $Poisson(\lambda)$ random variables.

\item
Let $R_G(b)$ denote the set of vertices in $N_G(b)$ which are not connected to any vertex in $N_G(a)$. Then given $\sF_n$, 
$$|R_G(b)|\sim Bin\Big(X_n^b, 1-(1-p_n)^{X_n^a}\Big) \text{ and }\mathbb{E}\left(\frac{|R_G(b)|}{X_n^b}\Big|\mathcal{F}\right)= 1-(1-p_n)^{X_n^a}\stackrel{\sP}{\rightarrow }0,$$ 
as $np_n^2\rightarrow 0$. Consequently, we have $|R_G(b)|=o_p(X_n^b)$. Defining $R_G(a)$ by symmetry we have $|R_G(a)|=o_p(X_n^a)$. Now, consider the 1-Lipschitz function $f$ on the (random) core-neighborhood of $G$ defined as follows: 
 $$f(x)=\left\{
\begin{array}{cc}
0 & \hbox{ if } x \in N_G(a)\backslash R_G(a),\\
1 & \hbox{ if } x \in R_G(a)\bigcup\{a\},\\
2 & \hbox{ if } x \in R_G(b)\bigcup\{b\},\\
3 & \hbox{ if } x \in N_G(b)\backslash R_G(b).\\
\end{array}
\right.
$$
This implies that 
\begin{align*}
\mathbb E_{b}(f)-\mathbb E_{a}(f)=& \frac{1+2|R_G(b)|+3(X_n^b-|R_G(b)|)}{1+X_n^b}-\frac{2+|R_G(a)|}{1+X_n^a}.
\end{align*}
The RHS converges to $3$ in probability. Thus, $W_n(a,b)\ge 3-o(1)$ with probability $1-o(1)$. Moreover, by definition $W_n(a,b)\le 3$ and the conclusion follows.

\item
As bipartite graphs are triangle free, from Equation \ref{eq:kappajosttrianglefree} we know that 
$\kappa(a,b)\le 0$. So, it suffices to prove only the lower bound. 
 
To this effect,  note that on the set $C_n:=\{X_n^a\ge np_n/2, X_n^b\ge np_n/2\}$ we have $(X_n^a\wedge X_n^b)p_n\ge np_n^2/2{\rightarrow}\infty$. Thus, if we denote by $H_G(a, b)$ the random bipartite graph induced by $N_G(a)\bigcup N_G(y)$, by Remark \ref{rematched} we have  
$$\mathbb{P}(H_G(a, b) \text{ does not have a matching of size } (1-\varepsilon)(|N_G(a)| \wedge |N_G(b)|)|\sF_n,C_n)\le \delta,$$ for $n$ large enough. Let $E_n$ be the event 
$\{H_G(a, b) \text{ does not have a matching of size } (1-\varepsilon)(|N_G(a)| \wedge |N_G(b)|)\}$.
This implies that for $n$ large enough
\begin{equation*}
\mathbb{P}(E_n)\le\mathbb{E}(\mathbb{P}(E_n|\sF_n,C_n))+\mathbb{P}(C_n^c)\le \delta+\mathbb P(\{X_n^a\wedge X_n^b\ge np_n/2\}^c).
\end{equation*}
As $(X_n^a\wedge X_n^b)/np_n\stackrel{\sP}{\rightarrow}1$, when $np_n\rightarrow \infty$, the above implies that with probability $1-o(1)$ there exists a bipartite matching between in $H_G(a, b)$ of size $(1-\varepsilon)(|N_G(a)| \wedge |N_G(b)|)$. This observation, together with Theoem \ref{th:matching}, would imply that with probability $1-o(1)$ we have  
 $$\lim_{n \rightarrow \infty}\kappa_n(a,b)\ge -2\Big(1-\frac{(X_n^a\wedge X_n^b)(1-\varepsilon)+1}{(1+X_n^a)\vee (1+X_n^b)}\Big)=-2\varepsilon.$$ As $\varepsilon$ is arbitrary, we have $\kappa_n(a, b)\stackrel{\sP}{\rightarrow}0$.

\end{enumerate}
\end{proof}

\subsection{Ricci Curvature of Erd\H os-R\'enyi random graphs}

Building on the techniques developed in the previous section, we now determine the limiting behavior of Ricci curvature for Erd\H os-R\'enyi random graphs $\mathbb G(n, p_n)$ in different regimes of $p$.

\begin{thm}\label{gnp}
Let $\mathbb G(n, p_n)$ be the distribution of a Erd\H os-R\'enyi random graph with vertex set $V_n$, conditioned on the edge $(a, b)$ being present. 
\begin{enumerate}[(a)]
\item

If $np_n\rightarrow 0$ then  with probability $1-o(1)$ the edge $(a,b)$ is isolated, and consequently 
$\kappa_n(a, b)=0$.

\item
If $np_n\rightarrow \lambda$ for $0<\lambda<\infty$, then $\kappa_n(a,b)\stackrel{\sD}{\rightarrow}-2\left(1-\frac{1}{1+X_1}-\frac{1}{1+X_2}\right)_+$,
where $X_1, X_2$ are independent $Poisson(\lambda)$ random variables, that is, the Ricci curvature converges in distribution to Ricci curvature of its limiting tree.

\item
If $np_n\rightarrow \infty, n^2p_n^3\rightarrow 0$, then $\kappa_n(a, b)\stackrel{\sP}{\rightarrow}-2.$

\item
If $n^2p_n^3\rightarrow\infty , np_n^2\rightarrow 0$, then $\kappa_n(a, b)\stackrel{\sP}{\rightarrow}-1$.

\item
If $np_n^2\rightarrow\infty, p_n\rightarrow 0$, then $\kappa_n(a, b)\stackrel{\sP}{\rightarrow }0.$

\item
If $p_n\rightarrow p$ with $0<p<1$, then $\kappa_n(a, b)\stackrel{\sP}{\rightarrow} p.$
\end{enumerate}
\end{thm}

\begin{proof}

Let $V_n=\{v_1, v_2,\ldots, v_n\}$ and $G\sim \mathbb G(n, p_n)$ be a random graph with vertex set $V_n$, conditioned on the edge $(a, b)$ being present, for 2 fixed vertices $a, b\in V_n$. 
For $\{i,j\}\subseteq \{1, 2, \ldots, n\}:=[n]$ define let $\delta_{ij}=1$ if $(v_i,v_j)\in E(G)$ and $0$, otherwise. 
Define $X^a_n:=\sum_{j\in [n]\backslash\{a,b\}}\delta_{aj}$ and $X_n^b:=\sum_{i\in [n]\backslash\{a,b\}}\delta_{ib}$. Clearly, $(X^a_n, X^b_n)$ are independent binomial random variables with parameters $(n-2,p_n)$, which are measurable with respect to $\sF_n:=\sigma(\delta_{aj},\delta_{ib }, i\in [n]\backslash\{a, b\}, j\in [n]\backslash\{a, b\})$.

\begin{enumerate}[(a)]
\item
By Markov's inequality we have $\mathbb{P}(\{X^a_n+X^b_n\neq 0\})\le \mathbb{E}(X^a_n+X^b_n)\le 2np_n\rightarrow 0$. So, the edge $(a, b)$ is isolated with probability $1-o(1)$, and  $\kappa_n(a, b)=0$.

\item
We will show that in this case there are no cycles of length 3, 4, or 5 supported on $(a, b)$ with probability $1-O(1/n)$. In this regard, note that
\begin{align*}
\mathbb{P}\left(\exists \text{ a 3-cycle supported on } (a, b)\right)\le & \sum_{j\in [n]\backslash\{a,b\}}\mathbb{E}(\delta_{aj}\delta_{bj})\le np_n^2=O(1/n),\\
\mathbb{P}\left(\exists \text{ a 4-cycle supported on } (a, b)\right)\le &\sum_{\{j, k\}\subset [n]\backslash\{a,b\}} \mathbb{E}(\delta_{aj}\delta_{jk}\delta_{bk})\le n^2p_n^3=O(1/n),\\
\mathbb{P}\left(\exists \text{ a 5-cycle supported on } (a, b)\right)\le &\sum_{\{j, k, l\}\subset [n]\backslash\{a,b\}}\mathbb{E}(\delta_{aj}\delta_{jk}\delta_{kl}\delta_{bl})\le n^3p_n^4=O(1/n).
\end{align*}
This implies that $\mathbb{P}(\exists \text{ a 3, 4, or 5 cycle supported on } (a, b))=O(1/n)$. Therefore, with probability $1-O(1/n)$, the girth of $G$ is at least 6, and Lemma \ref{gsix} implies that  $$\kappa_n(a,b)=-2\left(1-\frac{1}{1+X_n^a}-\frac{1}{1+X_n^b}\right)_+.$$ The conclusion follows on noting that $(X_n^a,X_n^b)\stackrel{\sD}\rightarrow (X_1, X_2)$, where $X_1, X_2$ are independent $Poisson(\lambda)$ random variables.

\item
As in the previous case, the probability of a triangle and a quadrilateral supported on $(a, b)$ is bounded by $np_n^2$ and $n^2p_n^3$, and so with probability $1-o(1)$ there are no 3 and 4-cycles supported on $(a,b)$. Let $Q_2(b)=\{x\in N_G(b): d_G(x, N_G(a))=2\}$. Note that 
$$|Q_2(b)|\Big|\sF_n \sim Bin(X_n^b, 1-(1-p_n^2)^{(n-4)X_n^a}) \text{ and } \mathbb{E}\Big(\frac{|Q_2(b)|}{X_n^b}\Big|\sF_n\Big)=1-(1-p_n^2)^{(n-4)X_n^a}\stackrel{\sP}{\rightarrow}0.$$

Now, consider the 1-Lipschitz function $f$ on the (random) core-neighborhood of $G$ defined as follows: 
$$f(x)=\left\{
\begin{array}{cc}
0 & \hbox{ if } x \in N_G(a),\\
2 & \hbox{ if } x \in Q_2(b)\bigcup\{b\},\\
3 & \hbox{ if } x \in N_G(b)\backslash Q_2(b),\\
1 & \hbox{ otherwise. }
\end{array}
\right.
$$
Then we have
\begin{align*}
\mathbb E_b(f)-\mathbb E_b(f)=\frac{1+2|Q_2(b)|+3(X_n^b-|Q_2(b)|)}{1+X_n^b}-\frac{2}{1+X_n^a}.
\end{align*}
The RHS converges to $3$ in probability, and the conclusion follows from arguments similar to the proof of part (c) of Theorem \ref{bp}.

\item
In this case there are no triangles supported on $(a, b)$ with probability $1-o(1)$. Define $R_1(a)=\{x\in N_G(a): d_G(x, N_G(b))=1\}$, and $R_1(b)$ similarly. By an argument similar to the previous case we have 
$R_1(a)=o_p(X_n^a)$ and $R_1(b)=o_p(X_n^b)$.  Now, for any $z_1\in N_G(a)\backslash R_1(a)$ and $z_2\in N_G(b)\backslash R_1(b)$ we have $d(z_1,z_2)\ge 2$, and so the function $f$ defined below is 1-Lipschitz:

$$f(x)=\left\{
\begin{array}{cc}
0 & \hbox{ if } x \in N_G(a)\bigcup \{a\},\\
2 & \hbox{ if } x \in N_G(b)\backslash R_1(b),\\
1 & \hbox{ otherwise. }
\end{array}
\right.
$$
Plugging in $f$ we have
$$\mathbb E_a(f)-\mathbb E_a(f)= \frac{|R_2(b)|+2(X_n^b-|R_2(b)|)}{1+X_n^b}-\frac{1}{1+X_n^a}.$$
The RHS converges to $2$ in probability, proving the upper bound for $\kappa_n(a, b)$.
\\

To prove the lower bound, assuming $\Delta_G(a, b)=\emptyset$ we claim that with probability $1-o(1)$ there exists a 2-matching between $N_G(a)$ and $N_G(b)$ of size $(1-\varepsilon)(X_n^a\wedge X_n^b)$. Assuming the claim is true by of Lemma \ref{bp-use}  
$$\kappa_n(a, b)\ge -2+\frac{2+(1-\varepsilon)X_n^a\wedge X_n^b}{(1+X_n^a)\vee (1+X_n^b)}.$$
The RHS converges to $-1-\varepsilon$ in probability, which would finish the proof of (d). 

To complete the proof we need to verify the claim, that is, construct a 2-matching between $N_G(a)$ and $N_G(b)$ 
in the subgraph $H_G(a,b)$. Define a new random graph $\sH(G)$ 
with $V(\sH(G))=N_G(a)\bigcup N_G(b)$ and an edge between $i\in N_G(a)$ and $j\in N_G(b)$ if and only if in the original random graph $G$ there is a path from $i$ to $j$ of length  $2$. 

By this construction, the probability that there is an edge between $i$ and $j$ in $\sH(G)$ is $\hat{p}_n:=1-(1-p_n^2)^{n-4}$. By  a Taylor's expansion we have
 $$|1-(1-p_n^2)^{n-4}-np_n^2|\le 4p_n^2+\frac{1}{2}n^2p_n^4(1+p_n^2)^{n-6}= np_n^2\left(\frac{4}{n}+\frac{1}{2}np_n^2 e^{np_n^2}\right)= np_n^2o(1),$$
and so for all large $n$ we have $\hat{p}_n\ge np_n^2/2$. Note that there is a 2-matching in $H_G(a, b)$ of size $(1-\varepsilon)(X_n^a\wedge X_n^b)$ if and only if $\sH(G)$ has a 1-matching of size $(1-\varepsilon)(X_n^a\wedge X_n^b)$. Since existence of a matching is a monotone property and the edges in $\sH(G)$ are positively correlated,  w.l.o.g. we may assume that  the edges in $\sH(G)$ are independent, as that would further reduce the probability of a  matching. Now, on the set  
$C_n:=\{X_n^a\ge np_n/2,X_n^b\ge np_n/2\}$ we have $(X_n^a\wedge X_n^b)\hat{p}_n\ge np_n\hat{p}_n/2\ge n^2p_n^3/4\rightarrow\infty$. Thus, by Lemma \ref{rematched} we have  $$\mathbb{P}( \sH(G)\text{ has no matching of size } (1-\varepsilon)(X_n^a\wedge X_n^b) |\sF_n ,C_n, |\Delta_G(a, b)|=0)\le \delta,$$ 
from which the claim follows, as $\mathbb P(C_n)\rightarrow 1$.

\item
In this case there may be triangles supported on $(a, b)$. Recall $R_G(a)=N_G(a)\backslash \Delta_G(a, b)$ and $R_G(b)=N_G(b)\backslash \Delta_G(a, b)$, and $H_G(a, b)$ is the subgraph of $G$ induced by the vertices in $R_G(a)\bigcup R_G(b)$. This implies that $|R_G(a)|=X_n^a-|\Delta_G(a, b)|, |R_G(b)|=X_n^b-|\Delta_G(a, b)|$. As $[(X_n^a\wedge X_n^b)-|\Delta_G(a, b)|]p_n\stackrel{\sP}{\rightarrow}\infty$, by Remark \ref{rematched} there exists a matching in $H_G(a, b)$ of size $(1-\varepsilon)(X_n^a\wedge X_n^b-\Delta_G(a, b))$ with probability $1-o(1)$, and so by Theoem \ref{th:matching} we have
\begin{align*}
\kappa_n(a, b)\ge &\frac{|\Delta_G(a, b)|}{(1+X_n^a)\vee (1+X_n^b)}-2\Big(1-\frac{(X_n^a\wedge X_n^b-\Delta_G(a, b))(1-\epsilon)+|\Delta_G(a, b)|}{(1+X_n^a)\vee (1+X_n^b)}\Big).
\end{align*}
The RHS converges to $-2\varepsilon$ in probability, as $|\Delta_G(x, y)|=o_p(X_n^a\wedge X_n^b)$. This proves that $\kappa_n(a, b)\geq 0$ in probability. Moreover, by Equation \ref{eq:kappajosttrianglefree} we have $\kappa_n(a, b)\leq \frac{|\Delta_G(x, y)|}{(X_n^a+1)\wedge (X_n^b+1)}$ which converges to $0$ in probability, completes the proof.

\item
 By a similar argument as in the previous case, there exists a matching between $R_G(a)$ and $R_G(b)$ of size $(1-\varepsilon)(X_n^a\wedge X_n^b-|\Delta_G(a, b)|)$, and so by Theoem \ref{th:matching} we have 
\begin{align*}
\kappa_n(a, b)\ge &\frac{|\Delta_G(a, b)|}{(1+X_n^a)(1+X_n^b)}-2\Big(1-\frac{(X_n^a\wedge X_n^b-|\Delta_G(a, b)|)(1-\varepsilon)+|\Delta_G(a, b)|}{(1+X_n^a)\vee (1+X_n^b)}\Big).
\end{align*}
The RHS converges to $p-2\varepsilon(1-p)$, which together with Equation \ref{eq:kappajosttrianglefree} implies that $\kappa_n(a, b)\stackrel{\sP}\rightarrow p$.

\end{enumerate}
\end{proof}

\small{\noindent{\bf Acknowledgement:} The authors are indebted to their advisor Persi Diaconis for introducing them to the problem and for his inspirational guidance. The authors thank Sourav Chatterjee, Austen Head, and Susan Holmes for useful discussions. The authors also thank the anonymous referees for valuable comments which improved the presentation of the paper.}

\end{document}